\documentclass[twoside,reqno]{amsart}
\usepackage[T1]{fontenc}
\usepackage[latin9]{inputenc}
\usepackage{geometry}
\usepackage{lscape}
\usepackage{setspace}
\usepackage{threeparttable}
\usepackage{float}
\usepackage{calc}
\usepackage{amsthm}
\usepackage{amssymb}
\usepackage{amsfonts}
\usepackage{amsmath}
\usepackage{algorithm}
\usepackage[noend]{algorithmic}
\usepackage{amssymb}
\usepackage{cite}
\usepackage{rotating}

\floatstyle{ruled}
\newfloat{algorithm}{tbp}{loa}
\floatname{algorithm}{Algorithm}

\numberwithin{equation}{section} 
\numberwithin{figure}{section} 
\theoremstyle{plain}
\theoremstyle{plain}
\newtheorem{thm}{Theorem}
\newtheorem{corollary}[thm]{Corollary}
\newtheorem{conjecture}[thm]{Conjecture}
  \theoremstyle{definition}
  \newtheorem{defi}[thm]{Definition}
 \theoremstyle{definition}
  \newtheorem{example}[thm]{Example}
  \theoremstyle{plain}
  \newtheorem{prop}[thm]{Proposition}
  \theoremstyle{plain}
  \newtheorem{lem}[thm]{Lemma}
  \theoremstyle{remark}
  \newtheorem{rem}[thm]{Remark}
  \theoremstyle{plain}
  \newtheorem{situation}[thm]{Situation}
  \theoremstyle{plain}
  
  \theoremstyle{remark}
  \newtheorem*{rem*}{Remark}
\newcommand\lc{\ensuremath{\mathrm{HC}}}
\newcommand\lm{\ensuremath{\mathrm{HT}}}
\newcommand\sm{\ensuremath{\mathrm{ST}}}

\newcommand\rp{\ensuremath{\mathcal{P}}}
\newcommand\mt{\ensuremath{\mathbf{T}}}

\newcommand\red{\ensuremath{\mathrm{red}}}

\newcommand\mf{\ensuremath{\mathbf{F}}}

\newcommand\ind{\ensuremath{\textrm{index}}}
\newcommand\poly{\ensuremath{\mathrm{poly}}}

\newcommand\lab{\ensuremath{\mathcal{S}}}
\newcommand\Sig{\ensuremath{\mathcal{\mathcal{S}}}}

\newcommand\K{\ensuremath{\mathcal{K}}}
\newcommand\Kx{\ensuremath{\mathcal{K}[\underline{x}]}}

\newcommand\lcm{\ensuremath{\textrm{lcm}}}

\begin{document}
\nocite{dlSingularBook2006}
\nocite{bdPolyBoRi2009}
\nocite{bSlimGB2005}
\nocite{bcpMagma}
\nocite{arsPhd2005}
\nocite{dgpsSingular311}
\bibliographystyle{plain}
\title{Modifying Faug\`{e}re's F5 Algorithm to ensure termination}
\author{Christian Eder}
\address{Department of Mathematics, TU Kaiserslautern, P.O. Box 3049\\
67653 Kaiserslautern, Germany}
\email{ederc@mathematik.uni-kl.de}
\urladdr{http://www.mathematik.uni-kl.de/~ederc/}
\author{Justin Gash}
\address{Justin Gash\\Department of Mathematics\\Franklin College\\101 Branigin Blvd.\\Franklin, IN 46131 USA}
\email{JGash@franklincollege.edu}
\urladdr{http://www.franklincollege.edu/about-fc/department-directory/individual?employee=000168429}
\author{John Perry}
\address{John Perry\\Department of Mathematics\\University of Southern Mississippi\\Hattiesburg, MS 39406-5045 USA}
\email{john.perry@usm.edu}
\urladdr{http://www.math.usm.edu/perry/}

\maketitle
\begin{abstract}
The structure of the F5 algorithm to compute Gr\"obner bases makes it very efficient.
However, it is not clear whether it terminates for all inputs, not even for ``regular sequences''.

This paper has two major parts.
In the first part, we describe in detail the difficulties related to a proof of termination.
In the second part, we explore three variants that ensure termination.
Two of these have appeared previously in dissertations,
and ensure termination by checking for a Gr\"obner basis using traditional criteria.
The third variant, F5+, identifies a degree bound
using a distinction between ``necessary'' and ``redundant'' critical pairs
that follows from the analysis in the first part.
Experimental evidence suggests this third approach is the most efficient of the three.
\end{abstract}
\section{Introduction}
The computation of a Gr\"obner basis is a central step in the solution of many problems
of computational algebra.
First described in 1965 by Bruno Buchberger
\cite{bGroebner1965}, researchers have proposed a number of important reformulations of his initial 
idea \cite{bGroebnerCriterion1970,
bGroebnerCriterion1979, gmInstallation1988, fF41999, bSlimGB2005,
bdPolyBoRi2009,mmtSyzygies1992}.
Faug\`ere's F5 Algorithm, published in 2002 \cite{fF52002Corrected}, is in many cases
the fastest, most efficient of these reformulations.
Due to its powerful criteria, the
algorithm computes very few zero-reductions,
and if the input is a so-called ``regular sequence'', it never reduces
a polynomial to zero (see Section~\ref{sec:basics} for basic definitions).
In general, reduction to zero is \emph{the} primary bottleneck
in the computation of a Gr\"obner basis;
moreover, many of the most interesting polynomial ideals are regular sequences.
It is thus no surprise that F5 has succeeded at computing many Gr\"obner bases
that were previously intractable \cite{fF52002Corrected, Fau05CStar}.

An open question surrounding the F5 algorithm regards termination.
In a traditional algorithm to compute a Gr\"obner basis,
the proof of termination follows from the algorithm's ability to exploit
the Noetherian property of polynomial rings:
each polynomial added to the basis $G$ expands the
ideal generated by the leading monomials of $G$,
and this can happen only a finite number of times.
In F5, however, the same criteria that detect reduction to zero
also lead the algorithm to add to $G$ polynomials
which do not expand the ideal of leading terms.
We call these polynomials \emph{redundant}.
Thus, although the general belief is that F5 terminates at least for regular sequences,
no proof of termination has yet appeared,
\emph{not even if the inputs are a regular sequence} (see Remark~\ref{rem:problemprooftermination}).
On the other hand, at least one system of polynomials has been proposed as examples of non-termination
(one in the source code accompanying~\cite{stF5Rev2005}),
but this system fails only on an incorrect implementation of F5.

Is it possible to modify F5 so as to ensure termination?
Since the problem of an infinite loop is due to the appearance of redundant
polynomials, one might be tempted simply to discard them.
Unfortunately, as we show in Section~\ref{sec:problem},
this breaks the algorithm's correctness.
Another approach is to supply, or compute, a degree bound,
and to terminate once this degree is reached.
Tight degree bounds are known for regular and ``semi-regular'' sequences~\cite{laz83,bfs05},
but not in general,
so for an arbitrary input it is more prudent to calculate a bound based on the data.
To that end,
\begin{itemize}
\item \cite{gashPhD2008} tests for zero-reductions of these redundant polynomials (Section~\ref{sub:f5t}); whereas
\item \cite{arsPhd2005} applies Buchberger's lcm criterion (or ``chain'' criterion) on critical pairs (Section~\ref{sub:F5Brs}).
\end{itemize}
These approaches rely exclusively on traditional criteria that are extrinsic to the F5 algorithm,
so they must interrupt the flow of the basic algorithm
to perform a non-trivial computation,
incurring an observable penalty to both time and memory.

This paper shows that it is possible to guarantee termination by relying primarily
on the criteria that are intrinsic to the F5 algorithm.
After a review of the ideas and the terminology in Section~\ref{sec:basics},
we show precisely in Theorem~\ref{lem:f5criticalrepresentation} of Section~\ref{sec:problem}
why one cannot merely discard the redundant polynomials \emph{in medio res}:
many of these redundant polynomials
are ``necessary'' for the algorithm's correctness.
Section~\ref{sec:f5+} uses this analysis to describe a new approach
that distinguishes between two types of critical pairs:
those that generate polynomials necessary for the Gr\"obner basis,
and those that generate polynomials ``only'' needed for the correctness of F5.
This distinction allows one to detect the point where all necessary data for the
Gr\"obner basis has been computed. 
We then show how to implement this approach in a manner
that incurs virtually no penalty to performance (Section~\ref{sub:termcorr}).
Section~\ref{sec:results} shows that this new variant, which we call F5+,
\begin{itemize}
\item computes a reasonably accurate degree bound for a general input,
\item relies primarily (and, in most observed cases, only) on criteria intrinsic to F5, and 
\item minimizes the penalty of computing a degree bound.
\end{itemize}
Section~\ref{sec:conclusion} leaves the reader with a conjecture that,
if true, could compute the degree bound even more precisely.

We assume the reader to be familiar with
\cite{fF52002Corrected}, as the modifications are described
using the pseudo code and the notations stated there.  
\section{Basics}\label{sec:basics}

Sections \ref{subsec:polybasics}--\ref{subsec:groebnerbasics} give a short review of notations and basics of polynomials
and Gr\"obner bases; Section \ref{subsec:f5basics} reviews the basic
ideas of F5.

For a more detailed introduction on non-F5 basics we refer the reader
to \cite{gpSingularBook2007}. Readers familiar with these topics 
may want to skim this section for notation and terminology.
\subsection{Polynomial basics}\label{subsec:polybasics}
Let $\K$ be a field, $\rp := \Kx$ the polynomial ring over $\K$ in the variables
$\underline{x} := (x_1,\ldots,x_n)$. Let $T$ denote the set of terms $\{x^\alpha\}
\subset \rp$, where $x^\alpha := \prod_{i=1}^n x_i^{\alpha_i}$ and $\alpha_i \in
\mathbb{N}$.

A \emph{polynomial} $p$ over $\K$ is a finite $\K$-linear combination of
terms, i.e. $p = \sum_{\alpha} a_\alpha x^\alpha\in\rp, a_\alpha \in \K$.
The \emph{degree of p} is the integer
$\deg(p) = \mathrm{max}\{\alpha_1+\dots+\alpha_n \mid a_\alpha \neq 0\}$ for $p
\neq 0$ and $\deg(p)=-1$ for $p=0$.

In this paper $>$ denotes a fixed admissible ordering on the terms
$T$. W.r.t. $>$ we can write any nonzero $p$ in a
unique way as
\[ p = a_\alpha x^\alpha + a_\beta x^\beta+\ldots+a_\gamma x^\gamma,\quad
x^\alpha > x^\beta > \cdots > x^\gamma\]
where $a_\alpha,a_\beta,\dots,a_\gamma \in \K\backslash\{0\}$.
We define the 
\emph{head term of p} $\lm(p)=x^\alpha$ and the \emph{head
coefficient of p} $\lc(p) = a_\alpha$.
\subsection{Gr\"obner basics}
\label{subsec:groebnerbasics}
We work with homogeneous ideals $I$ in $\rp$.
For any $S\subset \rp$ let $HT(S) := \langle \lm(p) \mid p \in S 
\backslash \{0\}\rangle$. A finite set $G$ is called a \emph{Gr\"obner
basis} of an ideal $I$ if $G \subset
I$ and $HT(I) = HT(G)$. Let $p \in \rp$.
If $p=0$ or there exist $\lambda_i \in \rp, q_i \in G$ such that $p = \sum_{i=1}^k \lambda_i q_i$
and $\lm(p) \geq \lm(\lambda_i q_i)$ for all nonzero $q_i$,
then we say that there exists a \emph{standard representation of $p$ w.r.t. $G$},
or that $p$ \emph{has a standard representation w.r.t $G$}. 
We generally omit the phrase ``w.r.t. $G$'' when it is clear from the context.

Let $p_i,p_j \in \rp$. We define the \emph{s-polynomial of the critical pair $(p_i,p_j)$} to be 
\[p_{ij} := \lc(p_j) \frac{\gamma_{ij}}{\lm(p_i)}p_i-
\lc(p_i) \frac{\gamma_{ij}}{\lm(p_j)}p_j \]
where $\gamma_{ij} := \lcm\left(\lm(p_i),\lm(p_j)\right)$.\\
\begin{thm} 
  \label{thm:gb}
  Let $I$ be an ideal in $\rp$ and $G \subset I$ finite. $G$ is a
  Gr\"obner basis of $I$ iff for all $p_i,p_j \in G$ $p_{ij}$ has a standard
  representation.
\end{thm}
\begin{proof}
  See Theorem 5.64 and Corollary 5.65 in \cite[pp. 219--221]{bwkGroebnerBases1993}.
\end{proof}

In addition to inventing the first algorithm to compute Gr\"obner bases,
Buchberger discovered two relatively efficient criteria
that imply when one can skip an s-polynomial reduction~\cite{bGroebner1965,bGroebnerCriterion1979}.
We will refer occasionally to the second of these criteria.

\begin{thm}[Buchberger's lcm criterion]\label{def:buch2}
  Let $G\subset\rp$ be finite, and $p_i,p_j,p_k\in\rp$.
  If
  \begin{itemize}
    \item[(A)] $\lm(p_k)\mid\lcm(\lm(p_i),\lm(p_j))$, and
    \item[(B)] $p_{ik}$ and $p_{jk}$ have standard representations w.r.t.~$G$,
  \end{itemize}
  then $p_{ij}$ also has a standard representation w.r.t.~$G$.
\end{thm}

In the homogeneous case one can define a {\it $d$-Gr\"obner basis} $G_d$
of an ideal $I$: This is a basis of $I$ for which all s-polynomials up
to degree $d$ have standard representations (cf. Definition 10.40 in
\cite[p. 473]{bwkGroebnerBases1993}).

The following definition is crucial for understanding the problem of
termination of F5.
\begin{defi}
  \label{def:redundantpoly}
  Let $G$ be a finite set of polynomials in $\rp$. We say that $p \in G$ is
  {\it redundant} if there exists an element $p'\in G$ such that $p'\neq p$ and $\lm(p') \mid
  \lm(p)$.
\end{defi}
\begin{rem}
  While computing a Gr\"obner basis, a Buchber\-ger-style algorithm
  does \emph{not} add polynomials that are redundant \emph{at the
  moment they are added to the basis}, although the addition of
  other polynomials to the basis later on may render them redundant. This
  ensures termination, as it expands the ideal of leading monomials, and $\rp$
  is Noetherian. However, F5 adds many elements
  that are redundant \emph{even when they are added to the basis};
  see Section \ref{sec:problem}.
\end{rem}
It is easy and effective to interreduce the elements of the initial ideal
before F5 starts, so that the input contains only non-redundant polynomials;
in all that follows, we assume that this is the case.
However, even this does not prevent F5 from generating redundant polynomials.

Finally, we denote by $\varphi(p,G)$ the normal form of $p$ with respect
to the Gr\"obner basis $G$.
\subsection{F5 basics}
\label{subsec:f5basics}

It is beyond the scope of this paper to delve into all the details of F5;
for a more detailed discussion we refer the reader to
\cite{fF52002Corrected}, \cite{eF5Criteria2008}, and \cite{epF5C2009}.
In particular, we do not consider the details of correctness for F5, which are
addressed from two different perspectives in
\cite{fF52002Corrected} and \cite{epF5C2009}.
This paper is concerned
with
showing that the algorithm can be modified so that termination is guaranteed,
\emph{and}
that the modification does not disrupt the correctness of
the algorithm.


In order to make the explanations more focused and concise,
we now adapt some basic definitions and notation of \cite{fF52002Corrected}.
Let $\mf_i$ be the $i$-th canonical generator of $\rp^m$.
Denote $\mt = \cup_{i=1}^m\mt_i$ where $\mt_i = \{t\mf_i \mid t \in T\}$ and $R = \mt \times \rp$.
Define $\prec$, the extension of $<$ to $\mt$, by $t\mf_i \prec u\mf_j$ iff
\begin{enumerate}
  \item $i>j$, or
  \item $i=j$ and $t<u$.
\end{enumerate}
It is easy to show that $\prec$ is a well-ordering of $\mt$,
which implies that there exists a minimal representation in terms of the generators.
\begin{defi}\label{def: signature}
  Let $p\in\rp$ and $t\in T$.
  We say that $t\mf_i$ is the \textit{signature of }$p$ if
  there exist $h_i,\ldots,h_m\in\rp$ such that each of the following holds:
  \begin{itemize}
    \item $p=\sum_{k=i}^m h_k f_k$ and $\lm(h_i)=t$, and
    \item for any $H_j,\ldots,H_m\in\rp$ such that $p=\sum_{k=j}^m H_k f_k$ and $H_j\neq 0$,
      we have $t\mf_i\preceq\lm(H_j)\mf_j$.
  \end{itemize}
\end{defi}
\begin{defi}\label{def: labpoly}
  Borrowing from \cite{stF5Rev2005}, we call the element 
  \[r=(t\mf_i,p) \in R\] 
  of \cite{fF52002Corrected} a \emph{labeled polynomial}.
  (It is referred to as the representation of a polynomial in \cite{fF52002Corrected}.)
  We also denote
  \begin{enumerate}
    \item the \textit{polynomial part of r} $\poly(r) = p$,
    \item the \textit{signature of r} $\Sig(r) = t\mf_i$, and
    \item the \textit{signature term of r} $\sm(r) = t$, and
    \item the \textit{index of r} $\ind(r) = i$.
  \end{enumerate}  
  Following \cite{fF52002Corrected}, we extend the following operators to $R$:
  \begin{enumerate}
    \item $\lm(r) = \lm(p)$.
    \item $\lc(r) = \lc(p)$.
    \item $\deg(r) = \deg(p)$.
  \end{enumerate}
Let $0 \neq c \in \K$, $\lambda \in T$, $r=(t\mf_i,p) \in R$. Then we define the
following operations on $R$ resp. $\mt$: 
\begin{enumerate}
    \item $cr = (t\mf_i,cp)$,
    \item $\lambda r = (\lambda t \mf_i,\lambda p)$,
    \item $\lambda(t\mf_i) = (\lambda t)\mf_i$.
\end{enumerate}
\end{defi}
\emph{Caveat lector: }Although we call $\lab(r)$ the signature of $r$ in Definition~\ref{def: labpoly},
it might not be the signature of $\poly(r)$ as defined in Definition~\ref{def: signature}.
If the input is non-regular, it can happen (and does) that F5 reduces an s-polynomial $r_{ij}$ to zero.
The reductions are all with respect to lower signatures, so we have\[
  \poly(r_{ij})=\sum_{k=1}^{\# G}h_k \poly(r_k)
\] where $h_k\neq 0$ implies that $\lab(\lm(h_k)\cdot r_k)\prec\lab(r_{ij})$.
The signature of $r_{ij}$ is thus no larger than $\max_{h_k\neq 0}\{\lab(\lm(h_k)\cdot r_k)\}$;
that is, the signature of $r_{ij}$ is strictly smaller than $\lab(r_{ij})$.

On the other hand, Propositions~\ref{prop:ST correct} and~\ref{prop: suff cond for correct sig}
show that the algorithm does try to ensure that $\lab(r)$ is the signature of $\poly(r)$.
The proof of Proposition~\ref{prop:ST correct} is evident from inspection of the algorithm.
\begin{prop}
  \label{prop:ST correct}
  Let the list $F=(f_1,\ldots,f_m)\in\rp^m$ be the input of F5.
  For any labeled polynomial $r=(t\mf_i,p)$, $t\in T$, $1\leq i\leq m$, computed by the algorithm, there exist
  $h_1,\ldots,h_m \in \rp$ such that 
\begin{enumerate}
  \item $p = h_1 f_1+ \ldots+h_m f_m$,
  \item $h_1=\ldots=h_{i-1}=0$, and
  \item $\sm(r) = \lm(h_i)=t$.
\end{enumerate}
\end{prop}

Let $G=\{r_1,\dots,r_{n_G}\}\subset\rp$. We denote 
$\poly(G) = \{\poly(r_1),\dots,\poly(r_{n_G})\}$.

\begin{defi}
  \label{def:labeledstandardrepresentation}
  Let $r,r_1,\dots,r_{n_G} \in R$, $G =
  \{r_1,\dots,r_{n_G}\}$.
  Assume $\poly(r)\neq 0$.
  We say that $r$ has a \textit{standard representation
  w.r.t. G} if there exist $\lambda_1,\ldots,\lambda_{n_G} \in \rp$ such that 
  \[\poly(r) = \sum_{i=1}^{n_G} \lambda_i \poly(r_i),\] 
  $\lm(r) \geq
  \lm(\lambda_i) \lm(r_i)$ for all $i$, and $\lab(r) \succ \lm(\lambda_i) \lab(r_i)$
  for all $i$ except possibly one, say $i_0$, where $\lab(r) = \lab(r_{i_0})$ and $\lambda_{i_0}=1$.
We generally omit the phrase ``w.r.t. $G$'' when it is clear from the context.
\end{defi}
\begin{rem}\label{rem: props of std rep of lab poly}
  The standard representation of a labeled polynomial $r$ has two properties:
  \begin{enumerate}
    \item The polynomial part of $r$ has a standard representation as defined in
  Section \ref{subsec:groebnerbasics}, and 
    \item the signatures of the multiples of the $r_i$ are not greater than
      the signature of $r$.
  \end{enumerate}
  This second property makes the standard representation of a labeled polynomial
  more restrictive than that of a polynomial. 
\end{rem}
\begin{prop}\label{prop: suff cond for correct sig}
  Let the list $F=(f_1,\ldots,f_m)\in\rp^m$ be the input of F5.
  For any labeled poly\-nomial $r$ that is computed by the algorithm,
  if $r$ does not have a standard representation w.r.t.~$G$,
  then $\lab(r)$ is the signature of $\poly(r)$.
\end{prop}
In other words, even if $\lab(r)$ is not the signature of $\poly(r)$,
the only time this can happen
is when $r$ already has a standard representation, so it need not be computed.
On the other hand, the converse is false: once the algorithm ceases to reduce $\poly(r)$,
$r$ \emph{does} have a standard representation, and $\lab(r)$ remains the signature of $\poly(r)$.
\begin{proof}
We show the contrapositive. Suppose that there exists some $r\in G$
such that $\lab(r)$ is not the signature of $p=\poly(r)$.
Of all the $r$ satisfying this property, choose one such that $\lab(r)$ is minimal.
Suppose $\lab(r)=t\mf_i$.
By hypothesis, we can find $h_j,\ldots,h_m\in\rp$ such that
\[p=\sum_{k=j}^m h_k f_k,\quad h_j\neq 0,\quad\textrm{ and }\quad i<j\textrm{ or }\left[i=j\textrm{ and }t>\lm(h_j)\right].\]
Is $\sum h_k f_k$ a standard representation of $r$ w.r.t.~$G$?
Probably not, but it is clear that for any $k$ such that $\lm(h_k)\lm(f_k)>\lm(p)$,
there exists $\ell$ such that $\lm(h_k)\lm(f_k)=\lm(h_\ell)\lm(f_\ell)$.
The signature of the corresponding $s$-poly\-nomial $p_{k\ell}$ is obviously smaller than $t\mf_i$,
so the hypothesis that $\lab(r)$ is minimal, along with inspection of the algorithm, implies that $r_{k\ell}$
has a standard representation w.r.t.~$G$.
Proceeding in this manner, we can rewrite $\sum h_k f_k$ repeatedly until we have
a standard representation of $r$ w.r.t.~$G$.
\end{proof}
\begin{defi}
  Let $r_i=(t_i \mf_k,p_i),r_j=(t_j \mf_\ell,p_j) \in R$.
  If $\frac{\gamma_{ij}}{\lm(r_i)}t_i\mf_k\neq\frac{\gamma_{ij}}{\lm(r_j)}t_j \mf_\ell$, then
  we define the \textit{s-polynomial of $r_i$ and $r_j$} by
  $r_{ij} := \left( m', p_{ij} \right)$
  where \[m'=\max_{\prec}\left\{\frac{\gamma_{ij}}{\lm(r_i)}t_i
  \mf_k,\frac{\gamma_{ij}}{\lm(r_j)}t_j \mf_\ell\right\}\]
  and $\gamma_{ij} = \lcm\left(\lm(r_i),\lm(r_j)\right)$.
\end{defi}
All polynomials are kept monic in F5; thus we always assume in the following that
$\lc(p_i) = \lc(p_j) = 1$ for $p_i \neq 0 \neq p_j$. Moreover we always assume
$\gamma_{ij}$ to denote the least common multiple of the head terms of the two
considered polynomial parts used to compute $r_{ij}$. 

Next we review the two criteria used in F5 to reject critical pairs which are not
needed for further computations.
\begin{defi}
  Let $G = \{r_1,\dots,r_{n_G}\}$ be a set of labeled polynomials, and $u_k\in T$.
  We say that $u_k r_k$ \textit{is detected by Faug\`ere's Criterion} if
  there exists $r \in G$ such that 
  \begin{enumerate}
    \item $\ind(r) > \ind(r_k)$ and 
    \item $\lm(r) \mid u_k \sm(r_k)$.
  \end{enumerate}
\end{defi}
\begin{defi}
  \label{def:FaugeresCrit}
  Let $G = \{r_1,\dots,r_{n_G}\}$ be a set of labeled polynomials, and $u_k\in T$.
  We say that $u_k r_k$ \textit{is detected by the Rewritten Criterion} if
  there exists $r_a \in G$ such that 
  \begin{enumerate}
      \item $\ind(r_a) = \ind(r_k)$,
      \item $a>k$, and
      \item $\sm(r_a) \mid u_k \sm(r_k)$.
  \end{enumerate}
\end{defi}
Next we can give the main theorem for the idea of F5.
Recall that we consider only homogeneous ideals.
\begin{thm}\label{thm:f5characterization}
  Let $I=\langle f_1,\dots,f_m\rangle$ be an ideal in $\rp$, and
  $G=\{r_1,\dots,r_{n_G}\}$ a set of
  labeled polynomials generated by the F5 algorithm (in that order)
  such that $f_i \in \poly(G)$ for $1\leq i \leq m$. Let $d\in\mathbb N$.
  Suppose that for any pair $r_i,r_j$ such that $\deg r_{ij}\leq d$ and $r_{ij} = u_i r_i - u_j r_j$,
  one of the following holds:
  \begin{enumerate}
    \item $u_k r_k$ is detected by Faug\`ere's Criterion for some $k\in\{i,j\}$, 
    \item $u_k r_k$ is detected by the Rewritten Criterion for some $k\in\{i,j\}$, or 
    \item $r_{ij}$ has a standard representation.
  \end{enumerate}
  Then $\poly(G)$ is a $d$-Gr\"obner basis of $I$.
\end{thm}

\begin{proof}
    See Theorem 1 in \cite{fF52002Corrected}, Theorem 3.4.2
    in \cite{gashPhD2008} and Theorem 21 in \cite{epF5C2009}.
\end{proof}
\begin{rem}\

  \begin{enumerate}
    \item \vspace{-1mm}Requiring a standard representation of a labeled polynomial
      is stricter than the criterion of Theorem
      \ref{thm:gb}, but when used carefully, any computational penalty imposed by this stronger condition is negligible
      when compared to the benefit from the two criteria it enables. 
    \item \label{rem:standardrepresentation}
      It is possible that $r_{ij}$ does not have a standard
      representation (cf. Proposition 17 in \cite{epF5C2009})
      at the time either Criterion rejects $(r_i,r_j)$.
      Since F5 computes the elements degree-by-degree, computations of the current degree 
      add new elements such that
      $r_{ij}$ has a standard representation w.r.t. the current Gr\"obner basis
      $\poly(G)$ before the next degree step is computed. Thus, at the end of each such
      step, we have computed a $d$-Gr\"obner basis of $I$.
  \end{enumerate}
\end{rem}
Next we give a small example which shows how the criteria work during the
computation of a Gr\"obner basis in F5.
\begin{example}
  Let $>$ be the degree reverse lexicographical ordering with $x>y>z$ on
  $\mathbb{Q}[x,y,z]$.  Let $I$ be the ideal generated by the following three
  polynomials:
  \begin{align*}
    p_1 &= xyz-y^2z,\\
    p_2 &= x^2-yz,\\
    p_3 &= y^2-xz. 
    \end{align*}
  Let the corresponding labeled polynomials be $r_i = (\mf_i,p_i)$. 
  For the input $F=(p_1,p_2,p_3)$, F5 computes a Gr\"obner basis of 
  $\langle p_2,p_3 \rangle$ as a first step:
  Since $\sm(r_{2,3}) = y^2 = \lm(r_3)$, $r_{2,3}$ is discarded by
  Faug\`ere's 
  Criterion. Thus $\{p_2,p_3\}$ is already a Gr\"obner basis of $\langle p_2,p_3
  \rangle$.

  \indent Next the Gr\"obner basis of $I$ is computed, i.e. $r_1$ enters the algorithm:
  Computing $r_{1,3}$ we get a new element: $r_4 = (y \mf_1, xz^3-yz^3)$.
  $r_{1,2}$ is not discarded by any criterion, but reduces to zero.
  Nevertheless its signature is recorded,\footnote{Failing to record the
  signature of a polynomial reduced to zero
  is an implementation error that can lead to an infinite loop.}
  thus we still have $\lab(r_{1,2}) = x \mf_1$ stored in the
  list of rules to check subsequent elements.

  Next check all s-polynomials with $r_4$ sorted by increasing signature: 
  \begin{enumerate}
    \item Since $\lab(r_{4,1}) = y^2 \mf_1$, $r_{4,1}$ is discarded by Faug\`ere's
      Criterion using $\lm(r_3) = y^2$.
    \item Since $\lab(r_{4,2}) = xy \mf_1$, $r_{4,2}$ is discarded by the
      Rewritten Criterion due to $\lab(r_{1,2}) = x \mf_1$, $r_{1,2}$ being
      computed after $r_4$.
    \item Since $\lab(r_{4,3}) = y^3 \mf_1$, $r_{4,3}$ is discarded by Faug\`ere's
      Criterion using $\lm(r_3) = y^2$.  
  \end{enumerate}
  The algorithm now concludes with $G=\{r_1,r_2,r_3,r_4\}$ where $\poly(G)$ is a
  Gr\"obner basis of $I$.
\end{example}

\section{Analysis of the problem}
\label{sec:problem}
The root of the problem lies
in the algorithm's reduction subalgorithms,
so Section~\ref{subsec:problem} reviews these in detail.
In Section~\ref{subsec:problemExample}, we show how the criteria force the reduction algorithms not only
to add redundant polynomials to the basis,
but to do so in a way that does not expand the ideal of leading monomials
(Example~\ref{example:FaugeresExample})!
One might try to modify the algorithm by simply discarding redundant polynomials,
but Section~\ref{sec:wheatchaff} shows that this breaks the algorithm's correctness.
This analysis will subsequently provide insights on how to solve the problem.

Throughout this section, let  the set of labeled
polynomials computed by F5 at a given moment be denoted
$G=\{r_1,\dots,r_{n_G}\}$. 
\subsection{F5's reduction algorithm}
\label{subsec:problem}
For convenience, let us summarize the reduction subalgorithms in some detail here.
Let $i$ be the current iteration index of F5. All newly computed labeled
polynomials $r$ satisfy $\ind(r) = i$. Let $G_{i+1}$ denote the set of
elements of $G$ with index $>i$.
We are interested in \texttt{Reduction},
\texttt{Top\-Reduction} and \texttt{IsReducible}. 
F5 sorts s-poly\-no\-mials by degree, and supplies to {\tt Reduction}
a set $F$ of s-polynomials of minimal degree $d$.
Let $r\in F$.
  \begin{enumerate}
    \item First, {\tt Reduction} replaces the polynomial part of $r$
      with its normal form with respect to $G_{i+1}$.
      This clearly does not affect the property $\lab(r)=\sm(r) \mf_i$.
      {\tt Reduction} then invokes {\tt TopReduction} on $r$.
    \item {\tt TopReduction} reduces $\poly(r)$ w.r.t. $G_i$, 
      but invokes {\tt IsReducible} to identify reducers.
      {\tt TopReduction} terminates whenever $\poly(r)=0$
      or {\tt IsReducible} finds no suitable reducers.
    \item {\tt IsReducible} checks all elements $r_\red \in G$ such that
      $\ind(r_\red)=i$.
      \begin{enumerate}
      \item If there exists $u_\red \in T$ such that 
      $u_\red \lm(r_\red) = \lm(r)$ then $u_\red\lab(r_\red)$ is 
      checked by both Faug\`{e}re's Criterion and the Rewritten Criterion. 
      \begin{enumerate}
        \item[$(\alpha)$] If neither criterion holds, the reduction takes place,
          but a further check is necessary to preserve
          $\lab(r)=\sm(r)\mf_i$.
          If $\lab(r) \succ u_\red \lab(r_\red)$, then it rewrites $\poly(r)$:
          \[r = \big(\lab(r),\poly(r)-u_\red \poly(r_\red)\big).\] 
          If $\lab(r) \prec u_\red \lab(r_\red)$, then $r$ is not changed, but a 
          new labeled polynomial is computed and added to $F$ for further reductions,
          \[r' = \big(u_\red \lab(r_\red), u_\red \poly(r_\red) - \poly(r)\big).\] 
          The algorithm adds $\lab(r')$ to the list of rules and continues with $r$.
        \item[$(\beta)$] \label{problematic case} If $u_\red r_\red$ is detected by
          one of the criteria, then the reduction
          does {\it not} take place, and the search for
          a reducer continues.
      \end{enumerate}
        \item If there is no possible reducer left to be checked then $r$ is added to $G$ if $\poly(r) \neq 0$.
      \end{enumerate}
  \end{enumerate}
  Note that if $\lab(r) = u_\red \lab(r_\red)$ then $u_\red r_\red$ 
  is rewrit\-able by $r$, thus Case (3)(a)($\beta$) avoids this situation.
  
\subsection{What is the problem with termination?}
\label{subsec:problemExample}
The difficulty with termination arises from Case (3)(a)($\beta$) above.
\begin{situation}\label{sit:whenaddredundant}
  Recall that $R_d$ is the set of labeled polynomials returned by \texttt{Reduction}
  and added to $G$.
  Suppose that $R_d\neq\emptyset$ and
  for every element $r\in R_d$,
  $\lm(\poly(r))$ is in the ideal generated by $\lm(\poly(G))$.
\end{situation}
\begin{example}\label{example:FaugeresExample}
Situation \ref{sit:whenaddredundant} is not a mere hypothetical:
as described in Section 3.5 of \cite{gashPhD2008},
an example appears in Section 8 of \cite{fF52002Corrected},
which computes a Gr\"obner basis of $(yz^3-x^2t^2,xz^2-y^2t,x^2y-z^2t)$.
Without repeating the details, at degree~7, F5 adds $r_8$ to $G$, with $\lm(r_8)=y^5t^2$.
At degree~8, however, {\tt Reduction} returns $R_8 = \{r_{10}\}$, with $\lm(r_{10})=y^6t^2$.
This is due to the fact that the reduction of $r_{10}$ by $yr_8$ is rejected by the
algorithm's criteria, and the reduction does {\it not} take place.
In other words, $r_{10}$ is added to $G$ even though $\poly(r_{10})$ is redundant in $\poly(G)$.
\end{example}
\begin{defi}
  \label{defi:redundant}
  A labeled polynomial $r$ computed in F5 is called {\it redundant} if,
  when {\tt Reduction} returns $r$, we have $\poly(r)$ redundant w.r.t. $\poly(G)$. 
\end{defi}
\begin{lem}
  \label{lem:nonredundantreducer}
  If $R_d$ satisfies Situation~\ref{sit:whenaddredundant} and $r\in R_d$,
  then we can find $r_k\in G$ such that $r_k$ is not redundant in $G$ and $\lm(r_k)\mid\lm(r)$.
\end{lem}
\begin{proof} 
  If a reducer $r_j$ of $r$ is redundant,
  then there has to exist another element $r_k$ such that $\lm(r_k)
  \mid \lm(r_j)$ and thus $\lm(r_k) \mid \lm(r)$.
  Follow this chain of divisibility down to the minimal degree;
  we need to show that there do not exist two polynomials $r_j$, $r_k$
  of minimal degree
  such that $\lm(r_j) = \lm(r_k)$. 
  Assume to the contrary that there exist $r_j,r_k\in G$ of minimal degree
  such that $\lm(r_j)=\lm(r_k)$.
  Clearly, the reduction of one by the other in {\tt IsReducible} was forbidden;
  without loss of generality, we may assume that $r_k$ was computed before $r_j$,
  so the reduction of $r_j$ by $r_k$ was forbidden.
  There are three possibilities:
  \begin{enumerate}
  \item If $\ind(r_k)>\ind(r_j)$,
    to the contrary, {\tt IsReducible} cannot interfere with this reduction,
    because such reductions are always carried out
    by the normal form computation in {\tt Reduction}.
  \item If $\lab(r_k)$ is rejected by the Rewritten Criterion, then there exists $r'$
    such that $\sm(r') \mid \sm(r_k)$, and $r'$ was computed after $r_k$.
    (That $r'$ was computed after $r_k$ follows from Definition~\ref{def:FaugeresCrit}, where $a>k$.)
    As F5 computes incrementally
    on the degree and $\sm(r') \mid \sm(r_k)$, it follows that $\deg(r') =
    \deg(r_k)$.
    Hence $\sm(r')=\sm(r_k)$.
    Thus the Rewritten Criterion would have rejected the computation of $r'$,
    again a contradiction.
  \item If $\lab(r_k)$ is rejected by Faug\`ere's Criterion,
    to the contrary, $r_k$ should not have been computed in the first place.
  \end{enumerate}
  Thus $\lm(r_j) \neq \lm(r_k)$. 
  It follows that we arrive at a non-redundant reducer after finitely many steps.
\end{proof}
\begin{lem}
  \label{lem:problem}
  Denote by $R_d$ the result of {\tt Reduction} at degree~$d$.
  There exists $m\in\mathbb N$ and an input $F=(f_1,\dots,f_m)$ and a degree $d$ such that
  if $\poly(G)$ is a $(d-1)$-Gr\"obner basis of $\langle f_1,\dots,f_m
  \rangle$, then 
  \begin{enumerate}
    \item[(A)] $R_d \neq \emptyset$, and
    \item[(B)] $HT(\poly(G \cup R_d)) = HT(\poly(G))$.
  \end{enumerate}
\end{lem}
\begin{proof}
  Such an input $F$ is given in Example \ref{example:FaugeresExample}:
  once reduction concludes for $d=8$,
  $\lm(r_8) \mid \lm(r_{10})$, so $\lm(\poly(G)) = \lm(\poly(G \cup R_8))$.
\end{proof}
\begin{rem} 
  \label{rem:problemprooftermination}
  In~\cite[Corollary 2]{fF52002Corrected}, it is argued that termination of F5
  follows from the (unproved) assertion that for any $d$,
  if no polynomial is reduced to zero,
  then $\lm(\poly(G))\neq\lm(\poly(G\cup R_d))$.
  But in Example~\ref{example:FaugeresExample},
  $\lm(\poly(G))=\lm(\poly(G\cup R_8))$,
  \emph{even though there was no reduction to zero!}
  Thus, Theorem~2 (and, by extension, Corollary~2) of \cite{fF52002Corrected} is incorrect:
  termination of F5 is unproved, \emph{even for regular sequences}, as
  there could be infinitely many steps where new redundant polynomials are added
  to $G$.
  By contrast, a Buchberger-style algorithm \emph{always} expands the monomial ideal
  when a polynomial does not reduce to zero; this ensures its termination.
\end{rem}
Having shown that there \emph{is} a problem with the proof of termination,
we can now turn our attention to devising a solution.

\label{sec:findingsolution}

\subsection{To sort the wheat from the chaff \dots isn't that easy!}
\label{sec:wheatchaff}
The failure of F5 to expand the ideal of leading monomials
raises the possibility of an infinite loop of redundant labeled polynomials.
However, we cannot ignore them.
\begin{example}
  Suppose we modify the algorithm to discard critical pairs
  with at least one redundant labeled polynomial.
  Consider a polynomial ring in a field of characteristic~7583.
  \begin{enumerate}
    \item For Katsura-5, the algorithm no longer terminates, but computes an 
      increasing list of polynomials with head terms $x_2^k x_4$ with 
      signatures $x_2x_3^kx_5x_6$ for $k \geq 1$. 
    \item For Cyclic-8, the algorithm terminates, but its output is not a 
      Gr\"obner basis!
  \end{enumerate}
\end{example}
How can critical pairs involving ``redundant'' polynomials can be necessary?
\begin{defi}
A critical pair $(r_i,r_j)$ is a {\it GB-critical pair} if neither $r_i$
nor $r_j$ is redundant. 
If a critical pair is not a GB-critical pair, then we call it an {\it F5-critical
pair}.
\end{defi}

We now come to the main theoretical result of this paper.

\begin{thm}
  \label{lem:f5criticalrepresentation}
  If $(r_i,r_j)$ is an F5-critical pair,
  then one of the following statements holds at the moment of creation of $r_{ij}$: 
  \begin{enumerate}
  \item[(A)] $\poly(r_{ij})$ already has a standard representation.
  \item[(B)] There exists a GB-critical pair $(r_k,r_\ell)$,
    a set $W\subset\{1,\ldots,n_G\}$,
    and terms $\lambda_w$ (for all $w\in W$) such that
    \begin{align}\poly(r_{ij}) = \poly(r_{k\ell}) + \sum_w \lambda_w \poly(r_w),\label{eq:redundantrewritten}\end{align}
    $\gamma_{ij} = \gamma_{k\ell}$ and $\gamma_{k\ell} > 
    \lambda_w \lm(r_w)$ for all $w$.
  \end{enumerate}
\end{thm}
Theorem~\ref{lem:f5criticalrepresentation} implies that an F5-critical pair
{\it might not} generate a redundant polynomial:
it might rewrite a GB-critical pair which is \emph{not} computed.
Suppose, for example, that the algorithm adds $r_i$ to $G$,
where $r_i$ is redundant with $r_k\in G$, perhaps because for $u\in T$ such that
$u\lm(r_k)=\lm(r_i)$, we have $\lab(u\cdot r_k)\succ\lab(r_i)$.
In this case, F5 will generate a new, reduced polynomial with the larger signature;
since the new polynomial has signature $\lab(u\cdot r_k)$,
the Rewritten Criterion will subsequently reject $u\cdot r_k$.
It is not uncommon that the algorithm later encounters some $r_\ell\in G$
where $r_{k\ell}$ is necessary for the Gr\"obner basis,
but $\lm(r_i)$ divides $\gamma_{k\ell}$.
In this case, the Rewritten Criterion forbids the algorithm from computing $r_{k\ell}$,
yet we can compute $r_{i\ell}$.
In terms of the Macaulay matrix~\cite{Macaulay02,laz83,fF52002Corrected},
the algorithm selects the row corresponding to $\frac{\gamma_{i\ell}}{\lm(r_i)}r_i$
instead of the row corresponding to $\frac{\gamma_{k\ell}}{\lm(r_k)}r_k$. 
Due to this choice,
the notions of ``redundant'' and ``necessary'' critical pairs 
are somewhat ambiguous in F5: is $r_i$ necessary to satisfy the properties of a Gr\"obner basis,
or to ensure correctness of the algorithm?
On the other hand, the notions of F5- and
GB-critical pairs are absolute.

To prove Theorem \ref{lem:f5criticalrepresentation}, we need the following
observation:
\begin{lem}\label{lem:no Spol of top-reduction}
Let $r_i,r_j\in G$ computed by F5, and assume that $\lm(r_j)\mid\lm(r_i)$.
Then {\tt Spol} does not generate an s-poly\-nomial for $(r_i,r_j)$.
\end{lem}
\begin{proof}
We have assumed that the input is interreduced, so $\poly(r_i)$ is not in the input.
Since $\lm(r_j) \mid \lm(r_i)$ there exists $u\in T$ such that
$u\lm(r_j)=\lm(r_i)$.
Since the reduction of $\poly(r_i)$ by $u\poly(r_j)$ was rejected, $u\lab(r_j)$ was detected by one of the criteria.
It will be detected again in {\tt CritPair} or {\tt Spol}.
Thus {\tt Spol} will not generate $r_{ij}$.
\end{proof}

\begin{proof}[Proof of Theorem \ref{lem:f5criticalrepresentation}]
  Assume that $r_i$ and $r_j$ are both redundant; the case where only 
  $r_i$ (resp.~$r_j$) is redundant is similar.
  By Lemma \ref{lem:nonredundantreducer} there exists for $r_i$ (resp.~$r_j$) at least one
  non-redundant reducer $r_k$ (resp.~$r_\ell$). 
  By Lemma~\ref{lem:no Spol of top-reduction}, we may assume
  that $r_i$ and $r_j$ are of degree smaller than $r_{ij}$.
  Using the fact that $\poly(G)$ is a $d$-Gr\"obner basis for $d=\max(\deg r_i, \deg r_j)$, we can write 
  \begin{align*}
    \poly(r_i) &= \lambda_{ik} \poly(r_k) + \sum_{u \in U} \lambda_u
    \poly(r_u)\\
    \poly(r_j) &= \lambda_{j\ell} \poly(r_\ell) + \sum_{v \in V} \lambda_v
    \poly(r_v),
  \end{align*}
  such that 
  \begin{align*}
      \lm(r_i) &= \lambda_{ik} \lm(r_k) > \lambda_u \lm(r_u)\;\rm{and}\\
     \lm(r_j) &= \lambda_{j\ell} \lm(r_\ell) > \lambda_v \lm(r_v)
  \end{align*}
  where $U,V \subset \{1,\dots,n_G\}$.
  As $\gamma_{k\ell} \mid \gamma_{ij}$, the representations of
  $\poly(r_i)$ and $\poly(r_j)$ above imply that there exists $\lambda \in
  T$ such that 
  \begin{align}
    \poly(r_{ij}) &= \frac{\gamma_{ij}}{\lm(r_i)} \poly(r_i) -
    \frac{\gamma_{ij}}{\lm(r_j)}\poly(r_j)
    \notag\\
     &= \lambda \poly(r_{k\ell}) + \sum_{w \in W} \lambda_w \poly(r_w) \label{formula1}  
  \end{align}
  where $W = U \cup V$ and $\lambda_w = \frac{\gamma_{ij}}{\lm(r_i)} \lambda_u$ for $w \in U
  \backslash
  V$, $\lambda_w =
  \frac{\gamma_{ij}}{\lm(r_j)} \lambda_v$ for $w \in V \backslash U$,
  and $\lambda_w = \frac{\gamma_{ij}}{\lm(r_i)}\lambda_u - \frac{\gamma_{ij}}{\lm(r_j)}\lambda_v$
  for $w \in U \cap V$.
  In Equation (\ref{formula1}) we have to distinguish two
  cases:
  \begin{enumerate}
    \item If $\lambda > 1$ then $\deg(r_{k\ell}) < \deg(r_{ij})$, thus
      $r_{k\ell}$ is already computed (or rewritten) using a lower degree
      computation, which has already finished. It follows that there exists a 
      standard representation of $\poly(r_{k\ell})$ and thus a standard representation
      of $\poly(r_{ij})$.
    \item If $\lambda = 1$ then $(A)$ holds if $\poly(r_{kl})$ is already computed by
      F5; otherwise $(B)$ holds.
  \end{enumerate}
\end{proof}
%

\label{subsec:isnteasy}
We can now explain why discarding redundant polynomials wreaks havoc in the algorithm.
\begin{situation}
\label{sit:f5gb}
Let $(r_i,r_j)$ be an F5-critical pair.
Suppose that all GB-critical pairs $(r_k,r_\ell)$
corresponding to case $(B)$ of Theorem~\ref{lem:f5criticalrepresentation}
are rejected by one of F5's criteria,
but lack a standard representation.
\end{situation}
Situation~\ref{sit:f5gb} is possible if, for example,
the Rewritten Criterion rejects all the $(r_k,r_\ell)$.
\begin{corollary}
  \label{lem:f5criticalnonredundant}
  In Situation \ref{sit:f5gb} it is necessary for the correctness of F5
  to compute a standard representation of $r_{ij}$.
\end{corollary}
\begin{proof}
  Since $\poly(r_{k\ell})$ lacks a standard representation,
  and the algorithm's criteria have rejected the pair $(r_k,r_\ell)$,
  then it is necessary to compute a standard representation of $r_{ij}$.
  Once the algorithm does so, we can rewrite~\eqref{eq:redundantrewritten}
  to obtain a standard representation of $\poly(r_{k\ell})$.
\end{proof}
In other words, ``redundant'' polynomials are necessary in F5.

\section{Variants that ensure termination}
\label{sec:terminating variants}

Since we cannot rely on an expanding monomial ideal,
a different approach to ensure termination could be to set or compute a degree bound.
Since a Gr\"obner basis is finite, its elements have a maximal degree.
Correspondingly, there exists
a maximal possible degree $d_{\mathrm{GB}}$ of a critical pair that generates a necessary polynomial.
Once we complete degree $d_{\mathrm{GB}}$, no new, non-redundant data for the Gr\"obner basis 
would be computed from the remaining pairs, so we can terminate the algorithm.
The problem lies with identifying $d_{\mathrm{GB}}$, which is rarely known
beforehand, if ever.\footnote{Another algorithm that computes a degree bound is MXL3~\cite{mxl3},
but its mechanism is designed for zero-dimensional systems over a field of characteristic~2.
It is not appropriate for the general case, whereas the approaches that we study here are.}

Before describing the new variant that follows from these ideas above,
we should review two known approaches, along with some drawbacks of each.

\subsection{F5t: Reduction to zero}\label{sub:f5t}
In \cite{gashPhD2008}, Gash suggests the following approach,
which re-introduces a limited amount of reduction to zero.
Once the degree of the polynomials exceeds $2M$,
where $M$ is the Macaulay bound for regular sequences~\cite{laz83,bfs05},
start storing redundant polynomials in a set $D$.
Whenever subalgorithm {\tt Reduction} returns a nonempty set $R_d$
that does not expand the ideal of leading monomials,
reduce all elements of $R_d$ completely w.r.t. $G \cup D$
and store any non-zero results in $D$ instead of adding them to $G$. 
Since complete reduction can destroy the relationship between a polynomial and its signature,
the rewrite rules that correspond to them are also deleted.
Subsequently, s-polynomials built using an element of $D$
are reduced without regard to criterion,
and those that do not reduce to zero are also added to $D$,
generating new critical pairs.
Gash called the resulting variant F5t.

One can identify four drawbacks of this approach:
\begin{enumerate}
\item The re-introduction of zero-reductions incurs a performance penalty.
    In Gash's experiments, this penalty was minimal, but these were performed on
    relatively small systems without many redundant polynomials.
    In some systems, such as Katsura-9, F5 works with hundreds of redundant polynomials.
\item It keeps track of two different lists for generating critical pairs and
    uses a completely new reduction process. An implementation must add a significant
    amount of complicated code
    beyond the original F5 algorithm.
\item It has to abandon some signatures due to the new, signature-corrupting
    reduction process. Thus, a large number of unnecessary critical pairs can be considered.
\item The use of $2M$ to control the size of $D$ is an imprecise, ad-hoc patch.  In some experiments from \cite{gashPhD2008}, F5t terminated on its own before polynomials reached degree $2M$; for other input systems, F5t yielded polynomials well beyond the $2M$ bound, and a higher bound would have been desirable.
\end{enumerate}

\subsection{F5B: Use Buchberger's lcm criterion}\label{sub:F5Brs}
In \cite{arsPhd2005}, Ars suggests using Buchberger's lcm criterion 
to determine a degree bound. 
\begin{itemize}
\item Initialize a global variable $d_B=0$ storing a degree. 
\item Keep a second list of critical pairs, $P^*$, used \emph{only}
    to determine a degree bound.
\item When adding new elements to $G$,
    store a copy of each critical pair not detected by Buchberger's lcm criterion in $P^*$.
    Remove any previously-stored pairs that are detected by Buchberger's lcm criterion,
    and store the highest degree of an element of $P^*$ in $d_B$.
\end{itemize}
If the degree of all critical pairs in $P$ exceeds $d_B$,
then a straightforward application of Buchberger's lcm criterion
implies that the algorithm has computed a Gr\"obner basis,
so it can terminate.
We call this variant F5B.

It is important to maintain the distinction between the two lists of critical pairs.
Otherwise, the correctness of the algorithm is no longer assured:
Buchberger's criteria ignore the signatures,
so $P^*$ lacks elements needed on account of Situation~\ref{sit:f5gb}.

While elegant, this approach has one clear drawback.
\emph{Every} critical pair is computed and checked twice:
once for Buchberger's lcm criterion, and again for the F5 criteria.
Although Faug\`ere's Criterion also checks for divisibility,
it checks only polynomials of smaller index,
whereas Buchberger's criterion checks \emph{all} polynomials,
and in most systems the number of polynomials of equal index
is much larger than the total of all polynomials having lower index.
Indeed, we will see in Section \ref{sec:results}
that this seemingly innocuous check can accumulate a significant time penalty.
This would be acceptable if the algorithm routinely used $d_B$ to terminate,
but F5 generally terminates from its own internal mechanisms
\emph{before} $d=d_B$!
Thus, except for pathological cases,
the penalty for this short-circuiting mechanism
is not compensated by a discernible benefit.

\subsection{F5+: Use F5's criteria on non-redundant critical pairs}
\label{sec:f5+}
We now describe a variant that uses information from F5 itself,
along with the theory developed in Section~\ref{sec:problem},
to reduce, if not eliminate, the penalty necessary to force termination.
We restate only those algorithms of~\cite{fF52002Corrected}
that differ from the original (and the differences are in fact minor).

The fundamental motivation of this approach stems from the fact that
a polynomial is redundant if and only if {\tt TopReduction} rejects a reductor
on account of one of the F5 criteria.
Understood correctly, this means that F5 ``knows'' at this point whether a polynomial is redundant.
We would like to ensure that it does not ``forget'' this fact.
As long as this information remains available to the algorithm,
identifying GB- and F5-critical pairs will be trivial.
Thus, our tasks are:
  \begin{enumerate}
    \item Modify the data structures to flag a labeled po\-lynomial
      as redundant or non-redundant.
    \item Use this flag to distinguish F5- and GB-critical pairs.
    \item Use the GB-critical pairs to decide when to terminate.
  \end{enumerate}
We address each of these in turn.


To distinguish between redundant and non-redun\-dant labeled polynomials,
we add a third, boolean field to the structure of a labeled polynomial.
We mark a redundant labeled polynomial with $b=1$, and
a non-redundant one with $b=0$.
Without loss of generality, the inputs are non-redundant,
so the first line of subalgorithm {\tt F5} can change to
\begin{algorithmic}
    \STATE $r_i:=(\mf_i,f_i,0) \in R \times \{0,1\}$
\end{algorithmic}
For all other labeled polynomials, the value of $b$ is set to~0 in algorithm \texttt{Spol}, then defined by the
behaviour of the {\tt Reduction} subalgorithm; see below. 


\label{subsec:implementationd0}
\label{subsec:tracking redundants}
The next step is to detect redundant polynomials; we do this in {\tt IsReducible}. 
In an unmodified F5, the return value of {\tt
IsReducible} is either a labeled polynomial $r_{i_j}$ (a polynomial that reduces
$r$) or $\emptyset$. The return value $\emptyset$ can
have two meanings:
\begin{enumerate}
  \item \label{item:nored} There exists no reducer of the input.
  \item \label{item:rejred} There exist reducers of the input, but their reductions are rejected.
\end{enumerate}
Algorithm \ref{alg:isreducible}, which replaces the original {\tt IsReducible} subalgorithm,
distinguishes these two possibilities by adding a boolean to the output:
$b=0$ in case (\ref{item:nored}) and $b=1$ otherwise.
We also need to modify subalgorithm {\tt TopReduction} to
use this new data; see Algorithm \ref{alg:topreduction}.

\label{sub:hybrid}
We now describe the main routine of the new variant,
which fulfills the following conditions:
\begin{enumerate}
\item Compute as low a degree bound as possible.
\item Minimize any penalty to the algorithm's performance. 
\end{enumerate}
An easy way to estimate $d_0$ would be to compute the highest degree of a GB-critical pair.
Although this would be correct, experience suggests that, in general,
it is much higher than necessary (see Table \ref{tab:timings} in Section \ref{sec:results}).
Instead, the new variant will use the criteria of the F5 algorithm
to identify GB-critical pairs that \emph{probably} reduce to zero.
How can we identify such pairs?
The following method seems intuitively correct:
\emph{when all GB-critical pairs are rejected by one of the F5 criteria}.

However, Situation~\ref{sit:f5gb} implies that this intuition may be \emph{in}correct.
Thus, \emph{once the algorithm reaches that degree} (and not earlier),
it uses Buchberger's lcm criterion to decide
whether the remaining GB-critical pairs reduce to zero.
If it can verify this, then the algorithm can terminate.

This differs from the approach of~\cite{arsPhd2005} in two important ways.
\begin{enumerate}
\item Rather than checking all pairs against the lcm criterion,
    it checks only GB-critical pairs that F5 also rejects as unnecessary.
    After all, it follows from Theorem~\ref{lem:f5criticalrepresentation}
    that F5-critical pairs can be necessary
    \emph{only if they substitute for a GB-critical pair}.
\item It checks the GB-critical pairs only once the F5 criteria suggest that it should terminate.
\end{enumerate}

\noindent
We call this variant F5+;
see Algorithm \ref{alg:f5hybrid}.
\begin{algorithm}[h!]
  \caption{{\tt IsReducible}}
  \label{alg:isreducible}
  \begin{algorithmic}
    \REQUIRE $\left\{\begin{array}{l}
        r_{i_0}\textrm{, a labeled polynomial of }R\\
        G=[r_{i_1},\dots,r_{i_r}]\\
        k\in\mathbb{N}\\
        \varphi\textrm{, a normal form}\end{array}
      \right.$
    \STATE $b := 0$
    \FOR{$j$ from 1 to $r$}
      \IF{$(u := \frac{\lm(r_{i_0})}{\lm(r_{i_j})} \in T)$}
        \IF{(neither criterion detects $(r_{i_0},r_{i_j}))$} 
          \RETURN $(r_{i_j},0)$
        \ELSE 
          \STATE $b := 1$
        \ENDIF
      \ENDIF
    \ENDFOR
    \RETURN $(\emptyset,b)$
  \end{algorithmic}
\end{algorithm}

\begin{algorithm}[h!]
  \caption{{\tt TopReduction}}
  \label{alg:topreduction}
  \begin{algorithmic}
    \REQUIRE $\left\{\begin{array}{l}
          r_{k_0}\textrm{, a labeled polynomial of }R\\
          G\textrm{, a list of elements of }R\\
          k\in\mathbb{N}\\
          \varphi\textrm{, a normal form}
        \end{array}\right.$
    \IF{$\poly(r_{k_0}) = 0$}
      \RETURN $(\emptyset,\emptyset)$
    \ENDIF
    \STATE $(r',b) := ${\tt IsReducible}$(r_{k_0},G,k,\varphi)$
    \IF{$r' = \emptyset$}
    \STATE $r_{k_0} := \big(\lab(r_{k_0}),\frac{1}{\lc(r_{k_0})}
    \poly(r_{k_0}),b\big)$
    \RETURN $(r_{k_0},\emptyset)$ 
    \ELSE
      \STATE $r_{k_1} = r'$
      \STATE $u := \frac{\lm(r_{k_0})}{\lm(r_{k_1})}$
      \IF{$u\lab(r_{k_1}) \prec \lab(r_{k_0})$}
        \STATE $r_{k_0} := \big(\lab(r_{k_0}), \poly(r_{k_0}) - u
        \poly(r_{k_1}),b\big)$ 
        \RETURN $(\emptyset,\{r_{k_0}\})$
      \ELSE
        \STATE $N := N+1$
        \STATE $r_N := \big(u\lab(r_{k_1}), u \poly(r_{k_1}) -
        \poly(r_{k_0}),b\big)$ 
        \STATE Add Rule $(r_N)$
        \RETURN $(\emptyset,\{r_N,r_{k_0}\})$
      \ENDIF
    \ENDIF    
  \end{algorithmic}
\end{algorithm}

\begin{algorithm}[h!]
  \caption{{\tt F5+}}
  \label{alg:f5hybrid}
  \begin{algorithmic}[1]
    \REQUIRE $\left\{\begin{array}{l}
          i\in\mathbb{N}\\
          f_i\in\Kx\\
          G_{i+1} \subset R\times \Kx,\textrm{ such that }\poly(G_{i+1})
            \textrm{ is a Gr\"obner basis of }\mathrm{Id}(f_{i+1},\ldots,f_m)
        \end{array}\right.$
    \STATE $r_i := (\mathbf{F}_i,f_i,0)$
    \STATE $\varphi_{i+1}:=\mathrm{NF}(.,\poly(G_{i+1}))$
    \STATE $G_i := G_{i+1}\cup\{r_i\}$
    \STATE \COMMENT{$P$ is the usual set of pairs; $P^*$ is the set of GB-pairs detected by the F5 criterion}
    \STATE $P := \emptyset$
    \STATE $P^* := \emptyset$
    \FOR{$r_j\in G_{i+1}$}\label{line:begin initial pairs}
      \STATE $p := \texttt{CritPair}(r_i,r_j,i,\varphi_{i+1})$
      \IF{$p = \emptyset$ and $r_j$ non-redundant}
        \STATE Add $(\lcm(\lm(\poly(r_i)),\lm(\poly(r_j))),r_i,r_j)$ to $P^*$
      \ELSE
        \STATE Add $p$ to $P$
      \ENDIF
    \ENDFOR\label{line:end initial pairs}
    \STATE Sort $P$ by degree
    \WHILE{$P\neq\emptyset$}\label{line:beginning of loop}
      \STATE $d := \deg(\mathrm{first}(P))$
      \STATE Discard from $P^*$ all pairs that are not of maximal degree
      \IF{$d\leq\max\{\deg(p):p\in P^*\}$ or $\exists p\in P^*$ that does not satisfy Buchberger's lcm criterion}
        \STATE $P_d := \{p\in P:\deg(p) = d\}$
        \STATE $P := P\backslash P_d$
        \STATE $F := \texttt{Spol}(P_d)$
        \STATE $R_d := \texttt{Reduction}(F,G_i,i,\varphi_{i+1})$
        \FOR {$r\in R_d$}\label{line:begin add pairs}
          \FOR{$r_j\in G_i$}
            \STATE $p := \texttt{CritPair}(r,r_j,i,\varphi_{i+1})$
            \IF{$p = \emptyset$ and $r,r_j$ both non-redundant}
              \STATE Add $(\lcm(\lm(\poly(r)),\lm(\poly(r_j))),r,r_j)$ to $P^*$
            \ELSE
              \STATE Add $p$ to $P$
            \ENDIF
          \ENDFOR
          \STATE $G_i := G_i\cup\{r\}$
        \ENDFOR\label{line:end add pairs}
        \STATE Sort $P$ by degree\label{line:end of loop}
      \ELSE
        \STATE $P:=\emptyset$
      \ENDIF
    \ENDWHILE
    \RETURN $G_i$
  \end{algorithmic}
\end{algorithm}

\begin{rem}
An implementation of F5+ has to take care when checking Buchberger's lcm criterion,
on account of the phenomenon of \emph{Buchberger triples}~\cite[p.~229]{bwkGroebnerBases1993}.
In \cite{arsPhd2005}, this is implemented similarly to the
``{\tt Update}'' algorithm of~\cite{bwkGroebnerBases1993,gmInstallation1988}.
The current F5+ takes a more traditional route;
it records all critical pairs that have generated s-polynomials.
The burden on memory is minimal.
\end{rem}

\subsection{Correctness and termination of F5+}
\label{sub:termcorr}
As a last step we have to show that F5+ terminates correctly.
\begin{thm}
  If F5+ terminates, the result is a Gr\"obner basis of the input.
\end{thm}
\begin{proof}
  This follows from 
  Buchberger's lcm criterion.
\end{proof}

\begin{thm}
  For a given homogeneous ideal $I$ as input,
  F5+ terminates after finitely many steps.
\end{thm}
\begin{proof}
  We first claim that after generating new critical pairs for $P$ in lines~\ref{line:begin add pairs}--\ref{line:end add pairs},
  F5+ satisfies $\# P < \infty$ at line~\ref{line:end of loop},
  and thus satisfies $\# P < \infty$ when the loop at line~\ref{line:beginning of loop} iterates anew.
  To show this, we will show that at any given degree $d$,
  the algorithm generates only finitely many polynomials and critical pairs.
  We proceed by induction on $d$; certainly $\# P<\infty$ after the loop
  in lines~\ref{line:begin initial pairs}--\ref{line:end initial pairs}. 
  Assume therefore that $\# P < \infty$ at line~\ref{line:beginning of loop}. 
  By the assumption that $\# P<\infty$, we have $\# P_d<\infty$,
  so {\tt Spol} generates only finitely many new polynomials.
  We now consider {\tt Reduction}; let $r\in ToDo$.
  \begin{enumerate}
    \item If 
      $\poly(r)=0$, then $r$ is effectively discarded; the algorithm does not add it to $G$,
      nor use it to generate new critical pairs.
    \item If 
      $\poly(r)\neq 0$,
      then {\tt IsReducible} checks for possible reducers:
      \begin{enumerate}
        \item If no reducer is returned, then $r$ is returned and added to
          $G$. All newly computed critical pairs generated by $r$ have degree $>d$;
          their number is finite because $G$ is currently finite.
        \item If {\tt IsReducible} returns $r_\red\in G$ such that there exists $u_\red \in T$
          satisfying $u_\red\lm(r_\red)=\lm(r)$ and
          $\lab(r) \succ u_\red \lab(r_\red)$, then $\poly(r)-u_\red\poly(r_\red)$
          replaces $\poly(r)$ in $r$, and $r$ is checked for further reductions.
          Note that $\lm(\poly(r))$ has decreased.
        \item If {\tt IsReducible} returns $r_\red\in G$ such that there exists $u_\red\in T$
          satisfying $u_\red\lm(r_\red)=\lm(r)$ but
          $u_\red \lab(r_\red) \succ \lab(r)$, then $r$ is not
          changed, but kept for further reduction checks. A new
          element
          $r'=(u_\red \lab(r_\red),\poly(r)-u_\red \poly(r_\red))$ is generated,
          and its signature $\lab(r')$ added to the list of rules. 
          Note that $\deg r'=\deg r$ and $\deg\sm(r')=\deg\sm(r)$.

          Only finitely many distinct reducers
          could lead to new elements $r'$.
          Since $\lab(r')$ was added to the list of rules, the Rewritten Criterion implies that $r_\red$ will not
          be chosen again as a reducer of $r$.
          There are only finitely many signatures of degree $d$,
          so only finitely many new elements can be added in this way.
      \end{enumerate}
  \end{enumerate}
  It follows that in each degree step only finitely many new polynomials are
  computed, so only finitely many new critical pairs are generated.
  Hence $\# P < \infty$ at line~\ref{line:end of loop}.

  To finish the proof we have to show that after finitely many steps, only 
  F5-critical pairs are left in $P$. 
  There can only be finitely many GB-critical pairs as their generating 
  labeled polynomials have to be non-redundant. Since $R$ is Noetherian, only
  finitely many non-redundant polynomials can be computed.

  Thus F5+ terminates after finitely many steps.
\end{proof}

\subsection{Experimental results}
\label{sec:results}
We implemented these variants in the {\sc Singular} 
kernel to compare performance.
(The F5 implementation in {\sc Singular} is still under development.) In Table \ref{table:1} we compare
timings and degree bounds for some examples.
All systems are homogeneous and
computed over a field of characteristic 32003.
The random systems are generated using the function {\tt sparseHomogIdeal} from {\tt random.lib} in {\sc Singular};
generating polynomials have a sparsity of 85-90\% and degrees $\leq 6$.
This data was recorded from a workstation
running Gentoo Linux on an Intel\textregistered \ Xeon\textregistered \ X5460 CPU
at 3.16GHz with 64 GB RAM.

Table~\ref{tab:timings} shows that
the tests for F5+ do
not slow it down significantly. 
But this is expected, since the modifications add trivial overhead,
and rely primarily on information that the algorithm already has available.

The computed degrees in Table~\ref{tab:timings} bear some discussion.
We have implemented F5+ in two different ways.
Both are the same in that they estimate the maximum necessary degree
by counting the maximal degree of a GB-critical pair
not discarded by the {\tt CritPair} subalgorithm.
However, one can implement a slightly more efficient {\tt CritPair} algorithm
by discarding pairs that pass Faug{\`e}re's Criterion, but are rewritable.
(The basic F5 checks the Rewritten Criterion only in subalgorithm {\tt Spol}.)
Thus one might compute a different maximal degree of $P^*$ in each case:
when {\tt CritPair} discards only those pairs detected by Faug\`ere's Criterion,
we designate the maximal degree of $P^*$ as $d_F$;
when {\tt CritPair} discards pairs detected by the Rewritten Criterion as well,
we designate the maximal degree of $P^*$ as $d_{FR}$.
We denote the degree where the original F5 terminates by $d_{\textrm{F5}}$,
and the maximal degree of a polynomial generated by $d_{\textrm{maxGB}}$.
Recall also that the maximal degree estimated by F5B is $d_B$ (Section~\ref{sub:F5Brs}).

It is always the case that $d_{\textrm{maxGB}}\leq d_{\textrm{F5}}$; indeed,
we will have $d_{\textrm{maxGB}}\leq d_A$ for any algorithm $A$ that computes a
Gr\"obner basis of a homogeneous system incrementally by degree.

On the other hand, it is always the case that $d_F,d_{FR}\leq d_{\textrm{F5}}$;
$d_{\textrm{F5}}$ counts F5-critical pairs as well as GB-critical pairs,
whereas $d_F,d_{FR}$ count only GB-critical pairs that are not rejected
by one or both of the F5 criteria.
Thus F5+ always starts its manual check for termination no later than F5 would terminate,
and sometimes terminates before F5.
For example, the termination mechanisms activate for F-855, Eco-10 and~-11, and Cyclic-8,
so F5B and F5+ both terminate at lower degree than F5.
With little to no penalty, F5+ terminates first,
but F5B terminates \emph{well after} F5 in spite of the lower degree!
Even in Katsura-$n$, where $d_{\textrm{maxGB}}=d_B < d_F=d_{FR}=d_{\textrm{F5}}$,
the termination mechanism of F5+ incurs almost no penalty,
so its timings are equivalent to those of F5,
whereas F5B is slower.
In other examples, such as Cyclic-7 and (4,5,12),
F5 and (therefore) F5+ terminate at or a little after the degree(s) predicted by $d_F$ and $d_{FR}$,
but \emph{before} reaching the maximal degree computed by $d_B$.

\begin{landscape}
\begin{threeparttable}
\centering
  \caption{Timings (in seconds) \& degrees of F5, F5B, and F5+}
  \centering\label{tab:timings}
  \begin{tabular}{|c|c|r|r|r|r|c|c|c|c|c|c|c|c|} \hline
    Examples\tnote{1}&regular?&F5&F5B&F5+&F5/F5B&F5/F5+&$d_{\textrm{maxGB}}$\tnote{2}&$d_{\textrm{F5}}$\tnote{3}&
      $d_{\textrm{GB-pair}}$\tnote{4}&$d_B$\tnote{5}&$d_F$\tnote{6}&
      $d_{FR}$\tnote{7}\\ \hline \hline
    {\tt Katsura-9}&yes&39.95&53.97&40.23&0.74&0.99&13&16&21&13&16&16\\ \hline
    {\tt Katsura-10}&yes&1,145.47&1,407.92&1,136.43&0.80&1.00&15&18&26&15&18&18\\ \hline
    {\tt F-855}&no&9,831.81&11,364.47&9,793.17&0.86&1.00&14&18&20&17&17&16\\ \hline
    {\tt Eco-10}&no&47.26&57.97&46.67&0.82&1.01&15&20&23&17&17&17\\ \hline
    {\tt Eco-11}&no&1,117.13&1,368.44&1,072.47&0.82&1.04&17&23&26&19&19&19\\ \hline
    {\tt Cyclic-7}&no&6.24&9.18&6.21&0.67&1.00&19&23&28&24&23&21\\ \hline
    {\tt Cyclic-8}&no&3,791.54&4,897.61&3,772.66&0.77&1.00&29&34&41&33&32&30\\ \hline \hline
    4,6,8&no&195.45&204.88&195.69&0.95&1.00&22&36&42&34&34&34\\ \hline
    5,4,8&yes&45.103&46.930&45.123&0.96&1.00&20&22&35&23&20&20\\ \hline
    6,4,8&no&46.180&46.880&46.247&0.99&1.00&20&20&34&22&20&20\\ \hline
    7,4,8&no&0.827&0.780&0.830&1.060&1.00&14&19&27&14&17&15\\ \hline
    8,3,8&no&122.972&126.816&123.000&0.97&1.00&22&37&35&26&31&29\\ \hline
    4,5,12&no&4.498&5.680&4.590&0.79&0.98&29&33&37&42&32&30\\ \hline
    6,5,12&yes&12.071&21.150&12.060&0.57&1.00&50&54&73&55&54&50\\ \hline
    8,4,12&no&46.122&47.613&47.750&0.97&0.97&27&35&44&30&34&29 \\ \hline
    12,4,12&no&14.413&14.897&14.360&0.97&1.00&42&55&60&43&53&43 \\ \hline
    4,3,16&yes&1.439&1.403&1.450&1.03&0.99&15&15&23&18&15&15 \\ \hline
    6,3,16&yes&36.300&37.136&36.300&0.98&1.00&10&14&23&15&14&13 \\ \hline
    8,3,16&yes&467.560&471.737&467.530&0.99&1.00&12&16&21&13&15&13 \\ \hline
    12,3,16&yes&210.327&206.441&210.311&1.02&1.00&21&25&34&20&24&23 \\ \hline
    4,3,20&yes&1.512&1.680&1.500&0.90&1.01&16&22&24&22&21&21\\ \hline
    6,4,20&no&1,142.433&1,327.540&1,144.370&0.86&1.00&27&37&39&29&35&31\\ \hline
    8,4,20&no&8.242&8.230&8.251&1.00&1.00&35&40&48&36&40&37 \\ \hline
    12,3,20&yes&0.650&0.693&0.650&0.94&1.00&22&26&34&27&26&23 \\ \hline
    16,3,20&no&2.054&2.060&2.050&1.00&1.00&26&26&41&27&26&26 \\ \hline
  \end{tabular}
    \begin{tablenotes}\footnotesize
    \item[1]{The notation $(a,b,c)$ denotes a random system of $a$ generators with maximal degree $b$ in a polynomial ring of $c$ variables.}
    \item[2]{maximal degree in GB}
    \item[3]{observed degree of termination of F5}
    \item[4]{maximal degree of a GB-critical pair}
    \item[5]{maximal degree estimated by Buchberger's lcm criterion; see Section~\ref{sub:F5Brs}}
    \item[6]{maximal degree of all GB-critical pairs not 
      detected by Faug\`ere's Criterion}
    \item[7]{maximal degree of all GB-critical pairs not 
      detected by Faug\`ere's Criterion \emph{or} the 
      Rewritten Criterion}
    \end{tablenotes}
    \label{table:1}
\end{threeparttable}
\end{landscape}

\section{Concluding remarks, and a conjecture}\label{sec:conclusion}
The new variant of F5 presented here is a straightforward solution to the problem of termination:
it distinguishes F5- and GB-critical pairs and tracks the highest degree of a necessary GB-critical pair.
Thus F5+ provides a self-generating, correct, and
efficient termination mechanism in case F5 does not terminate for some systems.
In practice, F5+ terminates before reaching the degree cutoff,
but it is not possible to test all systems, nor practical to determine \emph{a priori}
the precise degree of each Gr\"obner basis.
The question of whether F5, as presented in~\cite{fF52002Corrected},
terminates correctly on all systems, or even on all regular systems,
remains an important open question.


The following conjecture
arises from an examination of Table \ref{tab:timings}.

\begin{conjecture}\label{conj:degterm}
The F5 algorithm can terminate once all GB-critical pairs are rejected by the F5 criteria.
That is, it can terminate once $d=d_{\textrm{FR}}$.
\end{conjecture}

Conjecture~\ref{conj:degterm} is \emph{not} a Corollary of Theorem~\ref{thm:f5characterization}!
There, correctness follows only if \emph{all} critical pairs
are rejected by the algorithm: GB- \emph{and} F5-critical pairs.
Similarly, a proof of Conjecture~\ref{conj:degterm}
would imply that we could drop altogether the check of Buchberger's criteria.

If one could show that $d_{\textrm{maxGB}}\leq d_{FR}$,
Conjecture~\ref{conj:degterm} would follow immediately.
However, such a proof is non-trivial, and lies beyond the scope of this paper.
The conjecture may well be false even if we replace $d_{FR}$ by $d_F$,
although we have yet to encounter a counterexample.
The difficulty lies in the possibility that 
Situation~\ref{sit:f5gb} applies.

%
%

\section{Acknowledgements}
The authors wish to thank Martin Albrecht, Daniel Cabarcas, Gerhard Pfister and
Stefan Steidel for helpful discussions.
Moreover, we would also like to thank  the {\sc Singular} team at TU 
Kaiserslautern for their technical support.
We especially wish to thank the anonymous referees whose comments
improved the paper.

\end{document}